\documentclass[11pt,reqno]{amsart}
\usepackage{hyperref}
\usepackage{amssymb,amsmath,amsfonts,latexsym}
\usepackage{bm}
\usepackage{mathtools}
\usepackage{enumerate}
\usepackage{enumitem}
\usepackage{mathrsfs}
\usepackage{array,graphics,color}
\usepackage{csquotes}
\allowdisplaybreaks

\numberwithin{equation}{section}
\newtheorem{dfn}{Definition}[section]
\newtheorem{thm}[dfn]{Theorem}
\newtheorem{lma}[dfn]{Lemma}

\newtheorem{ppsn}[dfn]{Proposition}

\newtheorem{xmpl}[dfn]{Example}

\newtheorem{note}[dfn]{Note}

\DeclarePairedDelimiterX{\norm}[1]{\lVert}{\rVert}{#1}
\DeclarePairedDelimiterX{\bnorm}[1]{\big\lVert}{\big\rVert}{#1}
\DeclarePairedDelimiterX{\Bnorm}[1]{\Big\lVert}{\Big\rVert}{#1}

\begin{document}
	%\today
	
	\title[On sparse set topology using ideals in the space of reals]{On sparse set topology using ideals in the space of reals}
	
	\author[Debnath] {Indrajit Debnath}
	\address{Department of Mathematics, The University of Burdwan, Burdwan-713104, West Bengal, India}
	\email{ind31math@gmail.com}
	
	\author[Banerjee] {Amar Kumar Banerjee}
	\address{Department of Mathematics, The University of Burdwan, Burdwan-713104, West Bengal, India}
	\email{akbanerjee1971@gmail.com, akbanerjee@math.buruniv.ac.in}

	\subjclass[2020]{26E99, 54C30, 40A35}
	
	\keywords{Sparse set topology, density topology, ideal, $\mathcal{I}$-density topology}
	
	\begin{abstract}
		In this paper we have introduced the notion of $\mathcal{I}$-sparse set in the space of reals and explored some properties of the family of $\mathcal{I}$-sparse sets. Thereafter we have induced a topology namely $\mathcal{I}$-sparse set topology in the space of reals and it has been observed that this topology is finer than $\mathcal{I}-$density topology introduced by Banerjee and Debnath in \cite{banerjee 4}. We further studied some salient properties of this topology.
	\end{abstract}
	\maketitle
	
	\section{Introduction and preliminaries}
	
	In 1981 the classical notion of sparse sets in the space of reals appeared in the seminal work of D. N. Sarkhel and A. K. De \cite{sarkhel} while studying generalized integrals of the Perron type. It served as a generalization of such sets which have upper outer density zero. Subsequent development of this theory was done by S. Chakraborty and B. K. Lahiri in the setting of a topological group \cite{Chakraborty 1985,Chakraborty 1990} where they defined sparse sets by taking one of the equivalent conditions of sparse sets as in ( \cite{sarkhel}, Theorem 3.1). A topology was also generated with the help of sparse sets in \cite{Chakraborty 1990}. Sparse subsets of the real line were further investigated by T. Filipczak in \cite{Filipczak 1988, Filipczak 1993,Filipczak 2008}, by G. Horbaczewska in \cite{horbaczewska 2012} on the plane, by P. Das, M. M. Ar Rashid, A. K. Banerjee in \cite{das 2002,das,banerjee} in the setting of a metric space.

	Throughout the paper $\mathbb{R}$ stands for the set of all real numbers and $\mathbb{N}$ stands for the set of natural numbers. We shall use the notation $\mathcal{L}$ for the $\sigma$-algebra of Lebesgue measurable sets on $\mathbb{R}$, $\lambda^{\star}$ for outer Lebesgue measure and $\lambda$ for the Lebesgue measure on $\mathbb{R}$ \cite{halmos}. Wherever we write $\mathbb{R}$ it means that $\mathbb{R}$ is equipped with natural topology unless otherwise stated. We shall use the notation $2^{\mathbb{R}}$ to denote the power set of $\mathbb{R}$. The symmetric difference of two sets $A$ and $B$ be $(A  \setminus B)\cup (B \setminus A)$ and it is denoted by $A \triangle B$. By \enquote{a sequence of intervals $\{J_n\}_{n \in \mathbb{N}}$ about a point $p$} we mean $p \in \bigcap_{n \in \mathbb{N}}J_n$. The length of the interval $J_n$ will be denoted by $|J_n|$.

	Now, to introduce the notion of sparse set as in \cite{sarkhel} let us recall the classical definition of density point. For $E \in \mathcal{L}$ and a point $p \in \mathbb{R}$ we say the point $p$ is a classical density point \cite{Density topologies} of $E$ if and only if  $$\lim_{h \rightarrow{0+}} \frac{\lambda(E \cap [p-h,p+h])}{2h}=1 \ \mbox{where} \  \lambda \ \mbox{denotes the Lebesgue measure}.$$

	% Equivalently we can say the point $p \in \mathbb{R}$ is a classical density point of $E$ if and only if $$\lim_{h \rightarrow{0+}} \frac{\lambda((\mathbb{R} \setminus E) \cap [p-h,p+h])}{2h}=0.$$
	
	The set of all classical density point of $E$ is denoted by $\Phi (E)$. The collection $$\mathcal{T}_d=\{E \in \mathcal{L}: E \subset \Phi(E) \}$$ is a topology in the real line \cite{Density topologies} and it is called as the classical density topology. The celebrated Lebesgue Density Theorem is as follows.
	
	\begin{thm}\cite{Oxtoby} \label{e5}
		For any set $H \in \mathcal{L}$,	$$\lambda(H \triangle \Phi(H))=0.$$
		
	\end{thm}

	\begin{dfn}\cite{sarkhel}
		For $B \in \mathcal{L}$ and $p \in \mathbb{R}$ the right upper (resp. lower) density of $B$ at $p$ is defined as follows:
		$$d^{+}(B,p)=\limsup_{y\rightarrow{x^+}} \frac{\lambda(B \cap (x,y))}{|(x,y)|} \left(d_{+}(B,p)=\liminf_{y\rightarrow{x^+}} \frac{\lambda(B \cap (x,y))}{|(x,y)|}\right).$$
		
		If both right upper and lower densities are equal, then the common value is denoted by $d(B,p^{+})$ and it is called as the right density of $B$ at $p$.
	\end{dfn}
	
	In a similar manner the left densities $d^{-}(B,p), d_{-}(B,p)$ and $d(B,p^{-})$ are defined. If the right and left densities are equal, then the number $d(B,p)=d(B,p^{+})=d(B,p^{-})$ is the density of $B$ at $p$. If $d(B,p)=1$, then $p$ is a density point of B and if $d(B,p)=0$, then $p$ is a dispersion point of B.
	
	\begin{dfn}\cite{sarkhel}
		A set $A \subset \mathbb{R}$ is said to be sparse at a point $x$ on the right if for any $\varepsilon>0$ there exists $h>0$ such that every interval $(\alpha,\beta) \subset (x,x+h)$ with $\alpha-x<h(\beta -x)$ contains at least one point $y$ such that $\lambda^{\star}(A \cap (x,y))<\varepsilon (y-x)$. Sparseness on the left is defined similarly. A set $A$ is sparse at $x$ if it is sparse at $x$ on the right and on the left.
	\end{dfn}
	
	\begin{note}
		In \cite{sarkhel} it was shown that, if $x$ is a dispersion point of $A$, i.e., $d(A,x)=0$, then $A$ is sparse at $x$, and if $A$ is sparse at $x$ on the right, then $\overline{d}_{+}(A,x)<1$ and $\overline{d}_{+}(\mathbb{R} \setminus A,x)=1$.
	\end{note}
	
	\begin{thm}\cite{sarkhel}
		Let $B \in \mathcal{L}$ and $p \in \mathbb{R}$. Then the following conditions are equivalent:
	\end{thm}
	\begin{enumerate}
		\item $B$ is sparse at $p$ on the right;
		\item If $d^{+}(H,p)<1$, then $d^{+}(B \cup H ,p)<1$, for each $H \in \mathcal{L}$;
		\item If $d^{+}(B,p)<1$ and $d_{+}(B,p)=0$, then $d^{+}(B \cup H,p)<1$ and $d_{+}(B \cup H,p)=0$, for each $H \in \mathcal{L}$;
		\item If $d_{+}(H,p)=0$, then $d_{+}(B \cup H,p)=0$, for each $H \in \mathcal{L}$.
	\end{enumerate}
	
	A simple consequence of the above theorem is given the following proposition.
	
	\begin{ppsn}\cite{sarkhel}
		Let two sets $X,Y \in \mathcal{L}$ and $p \in \mathbb{R}$.
	\end{ppsn}
	\begin{enumerate}
		\item If $X$ is sparse at $p$ on the right, then $d^{+}(X,p)<1$ and $d_{+}(X,p)=0$;
		\item If $d(X,p^{+})=0$, then $X$ is sparse at $p$ on the right;
		\item If $X \subset Y$ and $Y$ is sparse at $p$ on the right, then $X$ is sparse at $p$ on the right;
		\item If $X$ and $Y$ are sparse at $p$ on the right, then $X \cup Y$ is sparse at $p$ on the right.   
	\end{enumerate}
	
	\begin{xmpl}\cite{sarkhel}
		Let $c>1$, $a_n=\frac{1}{c^{n^2+1}}$ and $b_n=\frac{1}{c^{n^2}}$ for $n \in \mathbb{N}$. The set $E=\bigcup_{n=1}^{\infty}(a_n,b_n)$ is sparse at $0$ on the right and $d^{+}(E,0)>1-\frac{1}{c}>0$.
	\end{xmpl}

	In order to incorporate the properties of larger class of non-convergent sequences the idea of convergence of real sequences was generalized to the notion of statistical convergence \cite{Fast,Schoenberg} followed by the idea of ideal convergence \cite{Kostyrko 2000}. Let us recall the definition of asymptotic density. For $K \subset \mathbb{N}$ we denote $K(n)$ to be the set $\{k \in K : k \leq n\}$ and $|K(n)|$ is the cardinality of $K(n)$. The asymptotic density of $K$ is defined by  $d(K)=\lim_{n\rightarrow{\infty}}\frac{|K(n)|}{n}$, provided the limit exists. The notion of asymptotic density was used to define the idea of statistical convergence by Fast \cite{Fast}, generalizing the idea of usual convergence of real sequences. A sequence $\{x_n\}_{n \in \mathbb{N}}$ of real numbers is said to be statistically convergent to $x_0$ if for given any $\epsilon > 0$ the set $K(\epsilon) = \{k\in \mathbb{N} : |x_k - x_0| \geq \epsilon\}$ has asymptotic density zero.

	After this pioneering work, the theory of statistical convergence of real sequences were generalized to the idea of $\mathcal{I}$-convergence of real sequences by P. Kostyrko et al., using the notion of ideal  $\mathcal{I}$ of subsets of $\mathbb{N}$, the set of natural numbers. We shall use the notation $2^{\mathbb{N}}$ to denote the power set of $\mathbb{N}$.
	
	\begin{dfn}\cite{Kostyrko 2000}
		A nonvoid class  $\mathcal{I} \subset 2^\mathbb{N}$ is called an ideal if $A,B \in  \mathcal{I}$ implies $A \cup B \in  \mathcal{I}$ and $A\in  \mathcal{I}, B\subset A$ imply $B\in \mathcal{I}$. Clearly $\{\phi\}$ and $2^{\mathbb{N}}$ are ideals of $\mathbb{N}$ which are called trivial ideals. An ideal is called nontrivial if it is not trivial.
	\end{dfn}

	$\mathcal{I}$ is admissible if it contains all singletons. By Fin we denote the ideal of all finite subsets of $\mathbb{N}$. It is easy to verify that the family $\mathcal{J}=\{A \subset \mathbb{N}: d(A)=0\}$ forms a non-trivial admissible ideal of subsets of $\mathbb{N}$. If $\mathcal{I}$ is a proper non-trivial ideal then the family of sets $\mathcal{F}(\mathcal{I}) = \{M\subset \mathbb{N} : \mathbb{N}\setminus M \in \mathcal{I}\}$ is a filter on $\mathbb{N}$ and it is called the filter associated with the ideal $\mathcal{I}$ of $\mathbb{N}$.

	\begin{dfn}\cite{Kostyrko 2000}
		A sequence $\{x_n\}_{n \in \mathbb{N}}$ of real numbers is said to be $\mathcal{I}$-convergent to $x_0$ if the set $K(\epsilon) = \{k\in \mathbb{N} : |x_k - x_0| \geq \epsilon\}$ belongs to $\mathcal{I}$ for any $\epsilon>0$.
	\end{dfn}
	
	If $\mathcal{I}=Fin$ then Fin-convergence coincides with the usual notion of convergence  of real numbers. If $\mathcal{I}$ is an admissible ideal then for a sequence of real numbers, the usual convergence implies $\mathcal{I}$-convergence with the same limit. Further many works were carried out in this direction by many authors \cite{banerjee 2,banerjee 3,Lahiri 2005}. K. Demirci  introduced the notion of  $\mathcal{I}$-limit superior and inferior of real sequence and proved several basic properties.
	
	\begin{dfn}\cite{Demirci}
		Let $\mathcal{I}$ be an admissible ideal in $\mathbb{N}$ and $x=\{x_n\}_{n \in \mathbb{N}}$ be a real sequence. Let,
		$B_x =\{b\in \mathbb{R} : \{k:x_k > b\}\notin \mathcal{I}\}$ and $A_x =\{a\in \mathbb{R} : \{k:x_k < a\}\notin \mathcal{I}\}$. Then the $\mathcal{I}$-limit superior of $x$ is given by,
		\begin{equation*}
			\mathcal{I}-\limsup  x  = \left\{
			\begin{array}{lr}
				\sup B_x & \text{if}\ B_x \neq \phi\\
				- \infty & \text{if}\ B_x = \phi
				
			\end{array} 
			\right.
		\end{equation*}

		and the $\mathcal{I}$-limit inferior of $x$ is given by,
		\begin{equation*}
			\mathcal{I}-\liminf  x = \left\{
			\begin{array}{lr}
				\inf A_x & \text{if}\ A_x \neq \phi\\
				\infty & \text{if}\ A_x = \phi
				
			\end{array} 
			\right.
		\end{equation*}
	\end{dfn}
	
	Further Lahiri and Das \cite{Lahiri 2003} carried out some more works in this direction. Throughout the paper the ideal $\mathcal{I}$ will always stand for a nontrivial admissible ideal of subsets of $\mathbb{N}$.

	Sparse set topology is the object of our interest and play a central role in our study. The goal of this paper is to investigate on $\mathcal{I}$-sparse set and $\mathcal{I}$-sparse set topology and its various properties. In section \ref{sec2} of this paper we first proved the finite sub-additivity and monotonicity of $\mathcal{I}$-density function \cite{banerjee 4} and then shown that $\mathcal{I}$-density topology on $\mathbb{R}$ is not separable.
	In section \ref{sec3} we have introduced the notion of $\mathcal{I}$-sparse set in the space of reals using the notion of ideal $\mathcal{I}$ of subsets of natural and explored some properties of the family of $\mathcal{I}$-sparse sets. In section \ref{sec4} we have induced a topology namely $\mathcal{I}$-sparse set topology in the space of reals and it has been observed that this topology is finer than $\mathcal{I}-$density topology introduced by Banerjee and Debnath in \cite{banerjee 4}. We further studied some salient properties of this topology.

	\section{$\mathcal{I}$-density}\label{sec2}
	
	In recent time A. K. Banerjee and I. Debnath have found a new way to generalize density topology using ideals in \cite{banerjee 4} which is as follows:
	
	\begin{dfn} \cite{banerjee 4}
		For $E\in \mathcal{L}$, $p\in \mathbb{R}$ and $n \in \mathbb{N}$ the upper $\mathcal{I}$-density of $E$ at the point $p$ denoted by $\mathcal{I}-d^-(p,E)$ and the lower $\mathcal{I}$-density of $E$ at the point $p$ denoted by $\mathcal{I}-d_-(p,E)$ are defined as follows: 
		Suppose $\{J_n\}_{n \in \mathbb{N}}$ be a sequence of closed intervals about $p$ such that
		\begin{center}
			$\mathscr{S}(J_n)=\{n\in \mathbb{N}:0<\lambda(J_n)<\frac{1}{n}\} \in \mathcal{F}(\mathcal{I}).$
		\end{center} 
		% Now we consider the following collection
		% \begin{center}
			% 	$\mathscr{C}_{p(\mathcal{I})}=\left\{\{J_n\}_{n \in \mathbb{N}}:\{J_n\}_{n \in \mathbb{N}} \mbox{be any sequence of closed intervals about} \ p \ \mbox{such that} \ \mathscr{S}(J_n) \in \mathcal{F}(\mathcal{I}) \right\}$
			% \end{center}
		For any such $\{J_n\}_{n \in \mathbb{N}}$ we take
		\begin{equation*}
			x_n = \frac{\lambda(J_n \cap E)}{\lambda(J_n)}  \ \mbox{for all} \ n \in \mathbb{N}.
		\end{equation*}
		Then $\{x_n\}_{n \in \mathbb{N}}$ is a sequence of non-negative real numbers. Now if 
		$$B_{x_k} =\{b\in \mathbb{R} : \{k:x_k > b\}\notin \mathcal{I}\}$$ and $$A_{x_k} =\{a\in \mathbb{R} : \{k:x_k < a\}\notin \mathcal{I}\}$$ we define,
		\begin{align*}
			\mathcal{I}-d^-(p,E) &= \sup \{ \sup B_{x_n}: \{J_n\}_{n \in \mathbb{N}} \ \mbox{such that} \ \mathscr{S}(J_n) \in \mathcal{F}(\mathcal{I})\} \\
			&= \sup \{\mathcal{I}-\limsup \ x_n: \{J_n\}_{n \in \mathbb{N}} \ \mbox{such that} \ \mathscr{S}(J_n) \in \mathcal{F}(\mathcal{I})\} 
		\end{align*}
		and 
		\begin{align*}
			\mathcal{I}-d_{-}(p,E) &= \inf \{ \inf A_{x_n}: \{J_n\}_{n \in \mathbb{N}} \ \mbox{such that} \ \mathscr{S}(J_n) \in \mathcal{F}(\mathcal{I})\} \\
			&= \inf \{\mathcal{I}-\liminf \ x_n: \{J_n\}_{n \in \mathbb{N}} \ \mbox{such that} \ \mathscr{S}(J_n) \in \mathcal{F}(\mathcal{I})\}. 
		\end{align*}
		
		In the above two expressions it is to be understood that $\{J_n\}_{n \in \mathbb{N}}$'s are closed intervals about the point $p$. Now, if $\mathcal{I}-d_-(p,E)=\mathcal{I}-d^-(p,E)$ then we denote the common value by $\mathcal{I}-d(p,E)$ which we call as $\mathcal{I}$-density of $E$ at the point $p$.
	\end{dfn}
	
	A point $p_0 \in \mathbb{R}$ is an $\mathcal{I}$-\textit{density point} of $E\in \mathcal{L}$ if $\mathcal{I}-d(p_0,E)=1$. 
	
	If a point $p_0\in \mathbb{R}$ is an $\mathcal{I}$-density point of the set $\mathbb{R}\setminus E$, then $p_0$ is an $\mathcal{I}$-\textit{dispersion point} of $E$.

	\begin{note}
		In \cite{banerjee 4} it was shown if $\mathcal{I}=\mathcal{I}_{fin}$ where $\mathcal{I}_{fin}$ is taken as the class of all finite subsets of $\mathbb{N}$, then Definition 3.1 coincides with the definition introduced by N. F. G. Martin \cite{Martin} of metric density.	The notion of $\mathcal{I}$-density point is more general than the notion of density point as the collection of intervals about the point $p$ considered in case of  $\mathcal{I}$-density is larger than that considered in case of classical density. For an example illustrating this fact see \cite{banerjee 4}.
	\end{note}
	
	\begin{thm}\cite{banerjee 4}
		Almost all points of an arbitrary measurable set $H$ are points of $\mathcal{I}$-density for $H$.
	\end{thm}
	
	\begin{lma} \label{s1} For each $p \in \mathbb{R}$ the set function $\mathcal{I}-d(p,.)$ defined on all subsets of $\mathbb{R}$ to unit interval $[0,1]$ is finitely subadditive and monotone.
		
	\end{lma}
	
	\begin{proof} Since $\mathcal{I}$-density exists at the point $p$, so upper $\mathcal{I}$-density and lower $\mathcal{I}$-density at the point $p$ are equal and equals to $\mathcal{I}$-density at the point $p$. Thus it is sufficient to show that $\mathcal{I}-d^{-}(p,.)$ is finitely subadditive and monotone. Let $\{Q_k\}_{k \in \mathbb{N}}$ be any sequence of closed intervals about the point $p$ such that $\mathscr{S}(Q_k) \in \mathcal{F}(\mathcal{I})$. For $H$ and $G$ any two measurable subsets of $\mathbb{R}$, let us take the real sequences $\{\zeta_n\}_{n \in \mathbb{N}},\{\eta_n\}_{n \in \mathbb{N}},\{\xi_n\}_{n \in \mathbb{N}}$ defined as 
		
		$$\zeta_n = \frac{\lambda(H \cap Q_n)}{\lambda (Q_n)}, \eta_n = \frac{\lambda(G \cap Q_n)}{\lambda (Q_n)}, \xi_n = \frac{\lambda((H \cup G) \cap Q_n)}{\lambda (Q_n)} \ \forall n \in \mathbb{N}.$$ 
		Then each of $\{\zeta_n\}_{n \in \mathbb{N}},\{\eta_n\}_{n \in \mathbb{N}},\{\xi_n\}_{n \in \mathbb{N}}$ is bounded hence $\mathcal{I}$-bounded. For any $n \in \mathscr{S}(Q_k)$, 
		$$\lambda((H\cup G)\cap Q_n) \leq \lambda(H\cap Q_n)+\lambda(G\cap Q_n).$$ So, $\xi_n \leq \zeta_n + \eta_n$ for $n \in \mathscr{S}(Q_k)$. So, it follows that 
		
		\begin{align*}
			\mathcal{I}-\limsup \xi_n & \leq \mathcal{I}-\limsup (\zeta_n+\eta_n)\\
			& \leq \mathcal{I}-\limsup \zeta_n + \mathscr{I}-\limsup \eta_n\\
			& \leq \sup \{\mathcal{I}-\limsup \zeta_n: \{Q_n\}_{n \in \mathbb{N}} \ \mbox{such that} \ \mathscr{S}(Q_n) \in \mathcal{F}(\mathcal{I})\}\\
			& \ \qquad + \sup \{\mathcal{I}-\limsup \eta_n :\{Q_n\}_{n \in \mathbb{N}} \ \mbox{such that} \ \mathscr{S}(Q_n) \in \mathcal{F}(\mathcal{I})\}\\
			& = \mathcal{I}-d^{-}(p,H)+ \mathcal{I}-d^{-}(p,G).
		\end{align*}

		Hence,
		\begin{equation}
			\begin{split}
				\mathcal{I}-d^{-}(p,H \cup G) &=\sup \{\mathcal{I}-\limsup \xi_n : \{Q_n\}_{n \in \mathbb{N}} \ \mbox{such that} \ \mathscr{S}(Q_n) \in \mathcal{F}(\mathcal{I})\}\\
				& \leq \sup \{\mathcal{I}-\limsup \zeta_n: \{Q_n\}_{n \in \mathbb{N}} \ \mbox{such that} \ \mathscr{S}(Q_n) \in \mathcal{F}(\mathcal{I})\}\\
				& \ \qquad + \sup \{\mathcal{I}-\limsup \eta_n :\{Q_n\}_{n \in \mathbb{N}} \ \mbox{such that} \ \mathscr{S}(Q_n) \in \mathcal{F}(\mathcal{I})\}\\
				& = \mathcal{I}-d^{-}(p,H)+ \mathcal{I}-d^{-}(p,G).
			\end{split}
		\end{equation}
		So, $\mathcal{I}-d^{-}(p,.)$ is finitely subadditive.
		
		Again if $H \subset G$, then since $ \frac{\lambda(H \cap Q_n)}{\lambda (Q_n)} \leq \frac{\lambda(G \cap Q_n)}{\lambda (Q_n)}$ we have 
		\begin{align*}
			\mathcal{I}-\limsup \zeta_n & \leq \mathcal{I}-\limsup \eta_n\\
			& \leq \sup \{\mathcal{I}-\limsup \eta_n: \{Q_n\}_{n \in \mathbb{N}} \ \mbox{such that} \ \mathscr{S}(Q_n) \in \mathcal{F}(\mathcal{I})\}\\
			& = \mathcal{I}-d^{-}(p,G).
		\end{align*}
		Hence,
		\begin{equation}
			\begin{split}
				\mathcal{I}-d^{-}(p,H) &=\sup \{\mathcal{I}-\limsup \zeta_n : \{Q_n\}_{n \in \mathbb{N}} \ \mbox{such that} \ \mathscr{S}(Q_n) \in \mathcal{F}(\mathcal{I})\}\\
				& \leq \sup \{\mathcal{I}-\limsup \eta_n: \{Q_n\}_{n \in \mathbb{N}} \ \mbox{such that} \ \mathscr{S}(Q_n) \in \mathcal{F}(\mathcal{I})\}\\
				& = \mathcal{I}-d^{-}(p,G).
			\end{split}
		\end{equation}
		So, $\mathcal{I}-d^{-}(p,.)$ is monotone. This completes the proof.
		
	\end{proof}
	
	\begin{dfn}\cite{banerjee 4}
		A measurable set $E \subset \mathbb{R}$ is $\mathcal{I}-d$ open iff  $\mathcal{I}-d_-(x,E)=1$ $\forall x \in E$
	\end{dfn}
	Let us take the collection $\mathbf{T}_\mathcal{I}=\{A \subset \mathbb{R}:$ $A$ is $\mathcal{I}-d$ open $\}$.

	\begin{thm} \cite{banerjee 4}
		The collection $\mathbf{T}_\mathcal{I}$ is a topology on $\mathbb{R}$.
	\end{thm} 
	
	\begin{dfn} \cite{banerjee 4}
		A point $x \in \mathbb{R}$ is a $\mathcal{I}-d$ limit point of a set $E \subset \mathbb{R}$ (not necessarily measurable) if and only if $\mathcal{I}-d^{-}(x,E)>0$ where instead of taking measure $m$ outer measure $m^{*}$ is taken.
	\end{dfn}
	
	\begin{thm} \cite{banerjee 4} \label{s99} In the space $(\mathbb{R}, \mathbf{T}_\mathcal{I})$ given any Lebesgue measurable set $E \subset \mathbb{R}$, $\lambda(E)=0$ if and only if $E$ is $\mathcal{I}-d$ closed and discrete.
	\end{thm}

	\begin{thm} \cite{banerjee 4} $(\mathbb{R}, \mathbf{T}_\mathcal{I})$ is a Hausdorff space.
	\end{thm}
	
	We now prove the following which will be needed later.
	
	\begin{thm}\label{s2} $(\mathbb{R}, \mathbf{T}_\mathcal{I})$ is not separable.
	\end{thm}
	
	\begin{proof}
		Let $C$ be any countable subset of $\mathbb{R}$. So clearly $\lambda(C)=0$. Thus, by Theroem \ref{s99}, $C$ is $\mathcal{I}-d$ closed. So, $\mathbb{R} \setminus C$ is an $\mathcal{I}-d$ open set that does not intersect $C$. Consequently, $C$ is not dense in $(\mathbb{R}, \mathbf{T}_{\mathcal{I}})$. Since choice of $C$ is arbitrary, so no countable set can be dense in $(\mathbb{R}, \mathbf{T}_{\mathcal{I}})$. Therefore, $(\mathbb{R}, \mathbf{T}_{\mathcal{I}})$ is not separable.
	\end{proof}
	
	\section{$\mathcal{I}$-Sparse set}\label{sec3}

	\begin{dfn} \cite{halmos}
		For any set $A \subseteq \mathbb{R}$, $E \in \mathcal{L}$ is said to be a measurable cover of $A$ if $A \subset E$ and for $F \in \mathcal{L}$ such that $F \subset E \setminus A$, $\lambda(F)=0$. 
	\end{dfn}

	\begin{dfn}\label{s3}
		A set $A \subseteq \mathbb{R}$ is said to be $\mathcal{I}$-sparse at a point $x \in \mathbb{R}$ if for every measurable set $B \subseteq \mathbb{R}$ with $\mathcal{I}-d(x,B)<1$, we have $\mathcal{I}-d(x,C \cup B)<1$, where $C$ is any measurable cover of $A$. The family of sets which are $\mathcal{I}$-sparse at $x$ is denoted by $\mathcal{S}_{\mathcal{I}}(x)$.
	\end{dfn}
	
	The lemma given below illustrates that the notion of $\mathcal{I}$-sparse sets is a generalization of the notion of sets of $\mathcal{I}$-density zero.
	
	\begin{lma}\label{s4}
		If $\mathcal{I}-d(x,A)=0$, then $A \in \mathcal{S}_{\mathcal{I}}(x)$.
	\end{lma}
	
	\begin{proof} Let $B$ be any Lebesgue measurable set such that $\mathcal{I}-d(x,B)<1$. By Lemma \ref{s1}, $$\mathcal{I}-d(x,A \cup B) \leq \mathcal{I}-d(x,A)+\mathcal{I}-d(x,B)=\mathcal{I}-d(x,B)<1.$$
		So, $A \in \mathcal{S}_{\mathcal{I}}(x)$.
		
	\end{proof}
	
	\begin{xmpl}
		We give an example showing that the converse of Lemma \ref{s4} is not true in the context of the set of real numbers. 
	\end{xmpl}

	Consider the set $E=\bigcup_{n \in \mathbb{Z}_{+}} (a_n,b_n)$ where $a_n=\frac{1}{c^{n^2+1}}$ and $b_n=\frac{1}{c^{n^2}}$, $c>1$. We are to show for $\mathcal{I}=\mathcal{I}_{d}$, the ideal of $\mathbb{N}$ consisting of all sets whose natural density are zero, $\mathcal{I}_d-d(0,E) \neq 0$. Let $x_n=\frac{\lambda(E \cap J_n)}{\lambda(J_n)}$ where $\{J_n\}_{n \in \mathbb{N}}$ is any sequence of closed intervals about the point $0$ such that $\mathscr{S}(J_n)=\left\{n\in \mathbb{N}:0<\lambda(J_n)<\frac{1}{n}\right\} \in \mathcal{F}(\mathcal{I}_d)$.\\
	Let $J_n=[l_{J_n},r_{J_n}]$ where $l_{J_n}$ denotes the left end point of $J_n$ and $r_{J_n}$ denotes the right end point of $J_n$. So given any large $n_0 \in \mathscr{S}(J_n)$ there exists some $k_0 \in \mathbb{N}$ such that $b_{k_0} < r_{J_{n_0}}$. Therefore, $E \cap J_{n_0} \supset (a_{k_0},b_{k_0})$. Thus,
	\begin{align*}
		\lambda(E \cap J_{n_0}) &> \lambda((a_{k_0},b_{k_0}))\\
		&= b_{k_0}-a_{k_0}\\
		&= b_{k_0}-\frac{1}{c} b_{k_0} \ \mbox{since} \ a_{k_0}=\frac{1}{c} b_{k_0}\\
		&= b_{k_0} (1-\frac{1}{c}).
	\end{align*}
	For $n \in \mathscr{S}(J_n)$ and $n > n_0$ where $n_0 \in \mathbb{N}$ there exists some $k_1 \in \mathbb{N}$ such that $\lambda(J_n) < b_{k_1}$. Since $(a_{k_0},b_{k_0}) \subset J_n$ and $\lambda(J_n)>b_{k_0}$ so $b_{k_1}>b_{k_0}$. Hence,
	\begin{align*}
		\frac{\lambda(E \cap J_n)}{\lambda(J_n)} &> \frac{b_{k_0}(1-\frac{1}{c})}{b_{k_1}}\\
		&= \frac{1}{c^{{k_0}^2-{k_1}^2}} (1-\frac{1}{c})\\
		&= \alpha  \ \mbox{(say)} \ \neq 0.
	\end{align*}
	Therefore, $\{n \in \mathbb{N}: x_n > \alpha\} \supset \mathscr{S}(J_n) \setminus \{1,2, \cdots, n_0\}$ and $\mathscr{S}(J_n) \in \mathcal{F}(\mathcal{I}_d)$. So, $\{n \in \mathbb{N}: x_n > \alpha\} \in \mathcal{F}(\mathcal{I}_d)$. Thus, $\{n \in \mathbb{N}: x_n > \alpha\} \notin \mathcal{I}_d$. Since, $\mathcal{I}_d - \limsup x_n =\sup \{b\in \mathbb{R} : \{k:x_k > b\}\notin \mathcal{I}_d\} > \alpha$. Consequently, $\mathcal{I}_d-d^-(0,E) = \sup \{\mathcal{I}_d-\limsup \ x_n: \{J_n\}_{n \in \mathbb{N}} \ \mbox{such that} \ \mathscr{S}(J_n) \in \mathcal{F}(\mathcal{I}_d)\}> \alpha $. Hence, $\mathcal{I}_d-d(0,E) \neq 0$.
	
	\begin{lma}\label{s5}
		If $E \subset A$ and $A \in \mathcal{S}_{\mathcal{I}}(x)$, then $E \in \mathcal{S}_{\mathcal{I}}(x)$.
	\end{lma}
	\begin{proof}
		Let $B$ be any Lebesgue measurable set such that $\mathcal{I}-d(x,B)<1$. Since $A \in \mathcal{S}_{\mathcal{I}}(x)$ so $\mathcal{I}-d(x,A' \cup B)<1$ where $A'$ is any measurable cover for $A$. Since $A \subset A'$, so $E \subset A'$. We claim that there exists a measurable cover $E'$ of $E$ such that $E \subset E' \subset A'$. Let $C'$ be any measurable cover of $E$. Now take $A' \cap C'$ which is a measurable set containing $E$ and contained in $C'$. So for any set $F \subset (A' \cap C') \setminus E$ we see $\lambda(F)=0$, since $(A' \cap C') \setminus E \subset C' \setminus E$ and $C'$ is a measurable cover of $E$. Thus $A' \cap C'$ is a measurable cover of $E$ contained in $A'$. We take $E'=A' \cap C'$. Now $E' \subset A'$ implies $E' \cup B \subset A' \cup B$ which implies, by Lemma \ref{s1}, that $\mathcal{I}-d(x,E' \cup B) \leq \mathcal{I}-d(x,A' \cup B)<1$. So, $E \in \mathcal{S}_{\mathcal{I}}(x)$. This proves the lemma.
	\end{proof}
	
	\begin{lma}\label{s6}
		If $E_1, E_2 \in \mathcal{S}_{\mathcal{I}}(x)$, then $E_1 \cup E_2 \in \mathcal{S}_{\mathcal{I}}(x)$.
	\end{lma}
	
	\begin{proof}
		Let $B$ be any Lebesgue measurable set such that $\mathcal{I}-d(x,B)<1$. Since $E_1 \in \mathcal{S}_{\mathcal{I}}(x)$ so there exists $E_1 '$, a measurable cover of $E_1$, such that $\mathcal{I}-d(x,E_1 ' \cup B)<1$. Since $E_2 \in \mathcal{S}_{\mathcal{I}}(x)$ so there exists $E_2 '$, a measurable cover for $E_2$, such that $\mathcal{I}-d(x,E_2 ' \cup E_1 ' \cup B)<1$. Moreover $E_1 ' \cup E_2 '$ is a measurable cover of $E_1 \cup E_2$. So, $E_1 \cup E_2 \in \mathcal{S}_{\mathcal{I}}(x)$. This proves the lemma.
	\end{proof}
	
	\begin{lma}
		$\mathcal{S}_{\mathcal{I}}(x)$ is a ring.
	\end{lma}
	
	\begin{proof}
		The proof follows from Lemma \ref{s5} and \ref{s6}.
	\end{proof}
	
	\begin{thm}
		Let $A \in \mathcal{S}_{\mathcal{I}}(x)$ and $F \subset \mathbb{R}$ be any Lebesgue measurable set such that $\mathcal{I}-d(x,F)=0$. Then $\mathcal{I}-d(x,A' \cup F)=0$ for any measurable cover $A'$ of $A$.
	\end{thm}
	
	\begin{proof} Let $A \in \mathcal{S}_{\mathcal{I}}(x)$ and $A'$ be any measurable cover of $A$. Since $\mathcal{I}-d(x,F)=0$, so by Lemma \ref{s6}, $\mathcal{I}-d(x,\mathbb{R} \setminus F)=1$. We are to show $\mathcal{I}-d(x,A' \cup F)=0$. Now first let us see
		\begin{align*}
			\mathcal{I}-d(x,A' \cup \{(\mathbb{R} \setminus A') \cap (\mathbb{R} \setminus F)\}) &= \mathcal{I}-d(x,A' \cup (\mathbb{R} \setminus F))\\
			& \geq \mathcal{I}-d(x,\mathbb{R} \setminus F) \ \mbox{by monotonicity}\\
			& =1.
		\end{align*}
		Therefore, 
		\begin{equation}\label{s7}
			\mathcal{I}-d(x,A' \cup \{(\mathbb{R} \setminus A') \cap (\mathbb{R} \setminus F)\})=1.
		\end{equation}
		We claim that $\mathcal{I}-d(x,(\mathbb{R} \setminus A') \cap (\mathbb{R} \setminus F))=1$. For if $ \mathcal{I}-d(x,(\mathbb{R} \setminus A') \cap (\mathbb{R} \setminus F))<1$ then since $A \in \mathcal{S}_{\mathcal{I}}(x)$, by definition of $\mathcal{I}$-sparse set, $ \mathcal{I}-d(x,A' \cup \{(\mathbb{R} \setminus A') \cap (\mathbb{R} \setminus F)\})<1$ which is a contradiction to equation \ref{s7}. Now, 
		\begin{align*}
			\mathcal{I}-d(x,(\mathbb{R} \setminus A') \cap (\mathbb{R} \setminus F)) &=  \mathcal{I}-d(x,\mathbb{R} \setminus (A' \cup F))\\
			&= 1- \mathcal{I}-d(x,(A' \cup F)).
		\end{align*}
		Thus, $\mathcal{I}-d(x,(A' \cup F))=1-1=0$ for any measurable cover $A'$ of $A$. This proves the theorem. 
		
	\end{proof}

	%\begin{crlre}
	%If $A \in \mathcal{S}_{\mathcal{I}}(x)$, then $\mathcal{I}-d(x,A')=0$ for any measurable cover $A'$ of $A$.
	%\end{crlre}
	%\begin{proof}
	%Since, $\mathcal{I}-d(x,\emptyset)=0$ so by Theorem 4.8, for any measurable cover $A'$ of $A$ we have $\mathcal{I}-d(x,A')=\mathcal{I}-d(x,A' \cup \emptyset)=0$.
	%\end{proof}

	\section{$\mathcal{I}$-Sparse set topology}\label{sec4}
	Let us consider the following collection
	$$\mathscr{U}=\{E \subset \mathbb{R}: \mathbb{R} \setminus E \in \mathcal{S}_{\mathcal{I}}(x) \ \mbox{for all} \ x \in E\}.$$
	
	\begin{thm}
		$\mathscr{U}$ is a topology for $\mathbb{R}$.
	\end{thm}
	\begin{proof}
		By voidness $\emptyset \in \mathscr{U}$. Since $\mathcal{I}-d(x,\emptyset)=0$ for all $x \in \mathbb{R}$ so, by Lemma \ref{s4}, $\emptyset \in \mathcal{S}_{\mathcal{I}}(x)$ which implies $\mathbb{R} \in \mathscr{U}$. Let $A_{\alpha} \in \mathscr{U}$ for each $\alpha \in \Lambda$ where $\Lambda$ is an arbitrary indexing set. We are to show that $\bigcup_{\alpha \in \Lambda}A_{\alpha} \in \mathscr{U}$. Let $A=\bigcup_{\alpha \in \Lambda}A_{\alpha}$ and $x \in A$. Then $x \in A_{\beta}$ for some $\beta \in \Lambda$. Since $A_{\beta} \in \mathscr{U}$, so $\mathbb{R}\setminus A_{\beta} \in \mathcal{S}_{\mathcal{I}}(x)$. Since $\mathbb{R} \setminus A \subset \mathbb{R} \setminus A_{\beta}$, it follows, by Lemma \ref{s5},  that $\mathbb{R} \setminus A \in \mathcal{S}_{\mathcal{I}}(x)$. Since $x \in A$ is arbitrarily chosen, so $A=\bigcup_{\alpha \in \Lambda}A_{\alpha} \in \mathscr{U}$. So, $\mathscr{U}$ is closed under arbitrary union. Next let us take $A_1, A_2 \in \mathscr{U}$ and $x \in A_1 \cap A_2 = D$. Then $x \in A_1$ and $x \in A_2$. So, $\mathbb{R} \setminus A_1, \mathbb{R} \setminus A_2 \in \mathcal{S}_{\mathcal{I}}(x)$. So by Lemma \ref{s6}, 
		$$\mathbb{R} \setminus D= \mathbb{R} \setminus (A_1 \cap A_2)=(\mathbb{R} \setminus A_1) \cup (\mathbb{R} \setminus A_2) \in \mathcal{S}_{\mathcal{I}}(x).$$
		Since $x \in D$ is arbitrary, so $D=A_1 \cap A_2 \in \mathscr{U}$. Thus $\mathscr{U}$ is closed under finite intersection. This proves the theorem.
	\end{proof}
	
	The topology $\mathscr{U}$ thus obtained is called the $\mathcal{I}$-Sparse set topology or in short the $\mathcal{I}/ \mathfrak{s}$-topology. The open (resp. closed) sets in this topology are called $\mathcal{I}/ \mathfrak{s}$-open (resp. $\mathcal{I} / \mathfrak{s}$-closed).
	
	\begin{thm}\label{s8}
		The $\mathcal{I} / \mathfrak{s}$-topology is finer than the $\mathcal{I}$-density topology.
	\end{thm}
	
	\begin{proof}
		Let $A$ be $\mathcal{I}-d$ open and $x \in A$. Then $\mathcal{I}-d_{-}(x,A)=1$ for all $x \in A$. Thus $\mathcal{I}-d(x,A)=1$ for all $x \in A$. This implies $\mathcal{I}-d_{-}(x,\mathbb{R} \setminus A)=0$ for all $x \in A$. So, by Lemma \ref{s4}, $\mathbb{R} \setminus A \in \mathcal{S}_{\mathcal{I}}(x)$. Thus, $A \in \mathscr{U}$. Since $x \in A$ is arbitrary, so $A$ is $\mathcal{I}/ \mathfrak{s}$ open. Consequently, $\mathcal{I} / \mathfrak{s}$-topology is finer than the $\mathcal{I}$-density topology. 
	\end{proof}
	
	\begin{thm}
		$(\mathbb{R}, \mathscr{U})$ is a Hausdorff space.
	\end{thm}
	\begin{proof}
		Since, by Theorem 4.2, $\mathcal{I} / \mathfrak{s}$-topology is finer than the $\mathcal{I}$-density topology and by Proposition 7.1 \cite{banerjee 4}, $(\mathbb{R}, \mathbf{T}_{\mathcal{I}})$ is Hausdorff space, so $(\mathbb{R}, \mathscr{U})$ is a Hausdorff space.
	\end{proof}
	
	\begin{thm}
		$(\mathbb{R}, \mathscr{U})$ is not separable.
	\end{thm}
	\begin{proof}
		Let, if possible, $(\mathbb{R}, \mathscr{U})$ be separable. Consider a countable set $C$ in $\mathbb{R}$ which is dense in $(\mathbb{R}, \mathscr{U})$. Thus for any $\mathcal{I}/ \mathfrak{s}$-open subset $V$ of $\mathbb{R}$, $V \cap C \neq \emptyset$. Since $\mathcal{I} / \mathfrak{s}$-topology is finer than the $\mathcal{I}$-density topology, so for the same set $C$ if we consider any $\mathcal{I}-d$ open set $W$ in $\mathbb{R}$ then $W$ is also $\mathcal{I} / \mathfrak{s}$-open set and $W \cap C \neq \emptyset$. So, $C$ is a countable set which is dense in $(\mathbb{R}, \mathbf{T}_{\mathcal{I}})$. Therefore, $(\mathbb{R}, \mathbf{T}_{\mathcal{I}})$ is separable. This is a contradiction to Theorem \ref{s2}. Therefore, $(\mathbb{R}, \mathscr{U})$ is not separable. 
	\end{proof}
	\begin{thm}
		For any Lebesgue measurable set $A$ and $x \in \mathbb{R}$, if $\mathcal{I}-d_{-}(x,A')>0$ where $A'$ is any measurable cover of $A$, then $x$ is an $\mathcal{I}/\mathfrak{s}$-limit point of $A$.
	\end{thm}
	\begin{proof}
		Suppose $x$ is not an $\mathcal{I}/\mathfrak{s}$-limit point of $A$. Then, without any loss of generality, there is a measurable $\mathcal{I}/\mathfrak{s}$-open set $U$ containing $x$ such that $U \cap (A \setminus \{x\})=\emptyset$. Therefore, $A \subset X \setminus \{U \setminus \{x\}\}=(X\setminus U) \cup \{x\}$. 
		
		Let $\{Q_k\}$ be any sequence of closed intervals about the point $x$ such that $\mathscr{S}(Q_k) \in \mathcal{F}(\mathcal{I})$. Since $\lambda(\{x\})=0$ so $\lambda^{\star}(A \cap Q_k) \leq \lambda^{\star}((X\setminus U) \cap Q_k)$. For any measurable cover $A'$ of $A$, $\lambda^{\star}(A)=\lambda(A')$. Thus, $\lambda (A' \cap Q_k) \leq \lambda ((X\setminus U) \cap Q_k)$. So, $$\frac{\lambda (A' \cap Q_k)}{\lambda(Q_k)} \leq \frac{\lambda ((X\setminus U) \cap Q_k)}{\lambda(Q_k)} \implies \liminf \frac{\lambda (A' \cap Q_k)}{\lambda(Q_k)} \leq \liminf \frac{\lambda ((X\setminus U) \cap Q_k)}{\lambda(Q_k)}.$$
		Therefore,
		\begin{align*}
			\mathcal{I}-d_{-}(x,A') &= \inf \left\{\liminf \frac{\lambda (A' \cap Q_k)}{\lambda(Q_k)}: \{Q_k\}_{k \in \mathbb{N}} \ \mbox{such that} \ \mathscr{S}(Q_k) \in \mathcal{F}(\mathcal{I}) \right \} \\
			& \leq \inf \left\{\liminf \frac{\lambda ((X\setminus U) \cap Q_k)}{\lambda(Q_k)}: \{Q_k\}_{k \in \mathbb{N}} \ \mbox{such that} \ \mathscr{S}(Q_k) \in \mathcal{F}(\mathcal{I}) \right \}\\
			&= \mathcal{I}-d_{-}(x,X\setminus U).
		\end{align*}
		Since $U$ is $\mathcal{I}/\mathfrak{s}$-open and $x \in U$, $X \setminus U \in \mathcal{S}_{\mathcal{I}}(x)$ and so $\mathcal{I}-d_{-}(x,X\setminus U)=0$ which implies $\mathcal{I}-d_{-}(x,A')=0$. Thus, by contraposition, the theorem is proved.
	\end{proof}
	
	\begin{thm}\label{s9}
		A necessary and sufficient condition that a measurable set $A$ is closed and discrete in the $\mathcal{I}/\mathfrak{s}$-topology is that $\lambda(A)=0$.
	\end{thm}
	\begin{proof}
		Assume that $\lambda(A)=0$. Then $\mathcal{I}-d(x,A)=0$ for all $x \in \mathbb{R}$. By Theorem 4.6 \cite{banerjee 4}, since $A$ is $\mathcal{I}-d$ closed and $\mathcal{I} / \mathfrak{s}$-topology is finer than the $\mathcal{I}$-density topology, $A$ is $\mathcal{I} / \mathfrak{s}$ closed. Also, by Theorem 4.6 \cite{banerjee 4}, $A$ is discrete in $\mathcal{I} / \mathfrak{s}$ topology.
		
		Conversely, assume that $A$ is $\mathcal{I} / \mathfrak{s}$-closed and discrete. Then, since $A$ has no limit points, $\mathcal{I}-d_{-}(x,A)=0$ for all $x \in \mathbb{R}$. If $A_1=\{x: x \in A \ \mbox{and} \ \mathcal{I}-d_{-}(x,A)=1\}$ and $A_2=\mathbb{R} \setminus A_1$, then $\lambda(A_2)=0$, by Theorem 3.2 \cite{banerjee 4}. Since $\mathcal{I}-d_{-}(x,A)=0$ for all $x \in \mathbb{R}$, so $A_1 =\emptyset$. Now, $A=A_1 \cup A_2$ and $\lambda(A)=\lambda(A_2)=0$. This proves the theorem.
	\end{proof}
	
	\begin{thm}\label{s10}
		If $A$ is any measurable subset of $\mathbb{R}$, then the set $G=\{x \in A: \mathbb{R} \setminus A  \in \mathcal{S}_{\mathcal{I}}(x) \}$ is $\mathcal{I}/\mathfrak{s}$-interior of $A$.
	\end{thm}
	
	\begin{proof}
		First we note that $G \subset A$. Now let $G_1=\{x \in A: \mathcal{I}-d(x,\mathbb{R} \setminus A)=0\}$. Since $A$ is measurable, so by  Theorem 3.2 \cite{banerjee 4}, $\lambda(A \setminus G_1)=0$. By lemma \ref{s4}, since $\mathcal{I}-d(x,\mathbb{R} \setminus A)=0$ implies $\mathbb{R} \setminus A 
		\in \mathcal{S}_{\mathcal{I}}(x)$, so $G_1 \subset G$. Thus, $A \setminus G \subset A \setminus G_1$ implies $\lambda(A \setminus G)=0$. 
		
		We need to show that $G$ is $\mathcal{I}/\mathfrak{s}$-open. Let $x \in G$ and $F \subset \mathbb{R}$ be a measurable set such that $\mathcal{I}-d(x,F)<1$. Since $x \in G$ and $\mathbb{R} \setminus A \in \mathcal{S}_{\mathcal{I}}(x)$, so $\mathcal{I}-d(x,F \cup (\mathbb{R} \setminus A))<1$. Now, $\mathbb{R} \setminus G=(\mathbb{R} \setminus A) \cup (A \setminus G)$. Consequently,
		\begin{align*}
			\mathcal{I}-d(x,F \cup (\mathbb{R} \setminus G)) &= \mathcal{I}-d(x,F \cup (\mathbb{R} \setminus A) \cup (A \setminus G))\\
			& \leq \mathcal{I}-d(x,F \cup (\mathbb{R} \setminus A))+ \mathcal{I}-d(x,A \setminus G), \ \mbox{by subadditivity lemma \ref{s1}}\\
			&= \mathcal{I}-d(x,F \cup (\mathbb{R} \setminus A)), \ \mbox{since} \ \lambda(A \setminus G)=0\\
			&<1.
		\end{align*}
		So, $\mathbb{R} \setminus G \in \mathcal{S}_{\mathcal{I}}(x)$. Thus, $G$ is $\mathcal{I}/\mathfrak{s}$-open.

		Next let $V$ be any $\mathcal{I}/\mathfrak{s}$-open set contained in $A$. By definition, $\mathbb{R} \setminus V \in \mathcal{S}_{\mathcal{I}}(x)$ for all $x \in V$. Since $\mathbb{R} \setminus A \subset \mathbb{R} \setminus V$, so by Lemma \ref{s5}, $\mathbb{R} \setminus A \in \mathcal{S}_{\mathcal{I}}(x)$ for all $x \in V$. As a result, $V \subset G$. Thus, $G$ is the largest $\mathcal{I}/\mathfrak{s}$-open set contained in $A$. This proves the theorem. 
	\end{proof}

	\begin{thm}
		For any measurable set $A$, if $\lambda(A)=0$ then $A$ is nowhere dense in the $\mathcal{I}/\mathfrak{s}$-topology.
	\end{thm}
	\begin{proof}
		Since $\lambda(A)=0$, so $A$ is $\mathcal{I}/\mathfrak{s}$-closed, by Theorem \ref{s9}. So, $A^{\kappa}=A$ where $A^{\kappa}$ denotes the $\mathcal{I}/\mathfrak{s}$-closure of a set $A \subset \mathbb{R}$. If the set of all $\mathcal{I}/\mathfrak{s}$-interior points of $A$ is denoted by $A^{\star}$ then, by Theorem \ref{s10}, $A^{\star} \subset \{x: x\in A \ \mbox{and} \ \mathcal{I}-d(x,A)=1\}= \mathcal{B} $(say). Since $\lambda(A)=0$ so $\mathcal{B}=\emptyset$ and consequently $A^{\star}=\emptyset$. So $(A^{\kappa})^{\star}=A^{\star}=\emptyset$. Therefore, $A$ is nowhere dense in the $\mathcal{I}/\mathfrak{s}$-topology. This proves the theorem. 
	\end{proof}
	
	\begin{lma}\label{s11}
		If $A$ is $\mathcal{I}/\mathfrak{s}$-open and $x \in A$, then $\mathcal{I}-d(x,A')=1$ for every measurable cover $A'$ of $A$.
	\end{lma}
	\begin{proof}
		Suppose that for any measurable cover $A'$ of $A$ and $x \in A$,  $\mathcal{I}-d(x,A')<1$. Since $A \in \mathscr{U}$ so $\mathbb{R} \setminus A \in \mathcal{S}_{\mathcal{I}}(x)$ for all $x \in A$. So by Definition \ref{s3} of $\mathcal{I}$-sparse set, $\mathcal{I}-d(x,A' \cup (\mathbb{R} \setminus A))<1$. This is a contradiction, since $A' \cup (\mathbb{R} \setminus A)= \mathbb{R}$. This proves the lemma.
	\end{proof}
	
	\begin{thm}
		$(\mathbb{R}, \mathscr{U})$ is not first countable.
	\end{thm}
	\begin{proof}
		Let us take an arbitrary point $b \in \mathbb{R}$. We are to show there does not exists any countable basis at $b$. If possible, let us assume that there exists $\{U_1,U_2,\cdots,U_n, \cdots\}$ a countable basis at $b$. Then we claim that each $U_i$ is an uncountable set. Because if $U_i$ is countable then for any measurable cover $U_i '$ of $U_i$, $\lambda(U_i ')=0$ and so for any $x \in U_i$, $\mathcal{I}-d(x,U_i ')<1$. But, by Lemma \ref{s11}, since each $U_i$ is $\mathcal{I}/\mathfrak{s}$-open, so $\mathcal{I}-d(x,U_i ')=1$ for $x \in U_i$. This is a contradiction. Thus each $U_i$ is an uncountable set. 
		
		Now let us take $b_1 \neq b$ from $U_1$, $b_2 \neq b,b_1$ from $U_2$, $b_3 \neq b,b_1,b_2$ from $U_3$ and so on. If we consider the set $\mathscr{B}=\{b_1,b_2,b_3,\cdots\}$ then $\mathscr{B}$ is countable and so $\lambda(\mathscr{B})=0$. Therefore, by Theorem \ref{s9}, $\mathscr{B}$ is $\mathcal{I}/\mathfrak{s}$-closed. Let $V$ be any $\mathcal{I}/\mathfrak{s}$-open set containing $b$. Consider the set $V \setminus \mathscr{B}$. This set is nonempty, since $V$ is uncountable. Also $V \setminus \mathscr{B}$ is an $\mathcal{I}/\mathfrak{s}$-open set containing $b$. Clearly no $U_i$ is contained in $V \setminus \mathscr{B}$. This shows that  $\{U_1,U_2,\cdots,U_n, \cdots\}$ cannot be a basis at $b$. Hence $(\mathbb{R}, \mathscr{U})$ is not first countable. This proves the theorem.
	\end{proof}

	\section*{Acknowledgements}
	\textit{The first author is thankful to The Council of Scientific and Industrial Research (CSIR), Government of India, for giving the award of Senior Research Fellowship (File no. 09/025(0277)/2019-EMR-I) during the tenure of preparation of this research paper.}

\end{document}